\documentclass[12pt]{amsart}

\usepackage{amscd,verbatim,enumerate,geometry,enumitem,pb-diagram,pb-xy,bbold,amssymb}

\usepackage{xy}
\xyoption{all}

\usepackage{hyperref}

\hypersetup{
pdfauthor={Olivier Haution, Alexander S. Merkurjev},
pdftitle={Connective K-theory and Adams operations},
pdfkeywords={},
hidelinks,
unicode
}

\numberwithin{equation}{section}

\DeclareMathAlphabet{\cat}{OT1}{cmss}{m}{sl}

\newtheorem{theorem}[equation]{Theorem}
\newtheorem{proposition}[equation]{Proposition}
\newtheorem{lemma}[equation]{Lemma}
\newtheorem{corollary}[equation]{Corollary}

\theoremstyle{definition}
\newtheorem{remark}[equation]{Remark}
\newtheorem{example}[equation]{Example}
\newtheorem{dfn}[equation]{Definition}

\newtheorem{notation}[equation]{Notation}

\newcommand{\tens}{\otimes}

\newcommand{\iso}{\stackrel{\sim}{\to}}

\newcommand{\id}{\mathrm{id}}

\newcommand{\colim}{\operatorname{colim}}
\newcommand{\CH}{\operatorname{CH}}
\newcommand{\CK}{\operatorname{CK}}
\renewcommand{\Im}{\operatorname{Im}}
\newcommand{\Ker}{\operatorname{Ker}}
\newcommand{\Coker}{\operatorname{Coker}}
\newcommand{\Colim}{\operatorname{colim}}
\newcommand{\B}{\operatorname{B\!}}
\newcommand{\tto}{\to\hspace{-0.5cm}\to}
\newcommand{\ch}{\operatorname{char}}
\newcommand{\Spec}{\operatorname{Spec}}
\newcommand{\Proj}{\operatorname{Proj}}
\newcommand{\SL}{\operatorname{SL}}
\renewcommand{\O}{\operatorname{O}}
\newcommand{\xra}{\xrightarrow}
\newcommand{\Coprod}{\operatornamewithlimits{\textstyle\coprod}}
\newcommand{\Prod}{\operatornamewithlimits{\textstyle\prod}}
\newcommand{\Sum}{\operatornamewithlimits{\textstyle\sum}}

\newcommand{\Tan}{T}
\newcommand{\cTan}{\Tan^{\vee}}
\newcommand{\Pp}{\mathbb{P}}
\newcommand{\Pu}{\mathbb{P}^1}
\newcommand{\Au}{\mathbb{A}^1}

\newcommand{\Z}{\mathbb{Z}}
\newcommand{\N}{\mathbb{N}}

\newcommand{\cM}{\mathcal M}
\newcommand{\cO}{\mathcal O}
\newcommand{\cZ}{\mathcal Z}
\newcommand{\cD}{\mathcal D}
\newcommand{\cC}{\mathcal C}
\newcommand{\cF}{\mathcal F}

\DeclareMathOperator{\characteristic}{char}

\title
[Connective $K$-theory and Adams operations] 
{Connective $K$-theory and Adams operations}

\author[O.~Haution]
{Olivier Haution}

\address{Olivier Haution\\
Mathematisches Institut der Ludwig-Maximilians-Universit\"at M\"unchen\\ Theresienstr.\ 39\\
D-80333 M\"unchen, Germany}

\email
{olivier.haution@gmail.com, {\it web page}: https://haution.github.io/}

\author[A.~Merkurjev]
{Alexander S. Merkurjev}

\address{Alexander Merkurjev\\
        Department of Mathematics \\
        University of California\\
        Los Angeles, CA \\
         USA}

\email
{merkurev@math.ucla.edu, {\it web page}: https://math.ucla.edu/\~{ }merkurev}

\begin{document}

\begin{abstract}
We investigate the relations between the Grothendieck group of coherent modules of an algebraic variety and its Chow group of algebraic cycles modulo rational equivalence. Those are in essence torsion phenomena, which we attempt to control by considering the action of the Adams operations on the Brown--Gersten--Quillen spectral sequence and related objects, such as connective $K_0$-theory. We provide elementary arguments whenever possible. As applications, we compute the connective $K_0$-theory of the following objects: (1) the variety of reduced norm one elements in a central division algebra of prime degree; (2) the classifying space of the split special orthogonal group of odd degree.
\end{abstract}

\thanks{The first author has been supported by DFG grant HA 7702/5-1 and Heisenberg fellowship HA 7702/4-1. The second author has been supported by the NSF grant DMS \#1801530.}

\maketitle

\section{Introduction}

The goal of the paper is to illustrate the usefulness of the connective $K_0$-groups of an algebraic variety $X$ and Adams operations for the study of relations between $K$-theory and the Chow groups of $X$.

For every integer $i$, denote $\cM_i(X)$ the abelian category of coherent $\cO_X$-modules with dimension of support at most $i$. We have a filtration $(\cM_i(X))$ of the category $\cM(X)$ of all coherent $\cO_X$-modules such that $\cM_i(X)=0$ if $i<0$ and $\cM_i(X)=\cM(X)$ if $i\geq \dim(X)$.

The $K$-groups of $\cM(X)$ are denoted $K'_n(X)$. The exact couple $(D_{r,s}, E_{r,s})$ of homological type with
\[
D_{r,s}^1=K_{r+s}(\cM_r(X))\quad \text{and} \quad E_{r,s}^1=\Coprod_{x\in X_{(r)}}K_{r+s}F(x),
\]
where $X_{(r)}$ denotes the set of points in $X$ of dimension $r$, yields the Brown-Gersten-Quillen (BGQ) spectral sequence
\[
\Coprod_{x\in X_{(r)}}K_{r+s}F(x)\Rightarrow K'_{r+s}(X)
\]
with respect to the topological filtration $K'_n(X)_{(i)}=\Im(K_n(\cM_i(X))\to K_n(\cM(X)))$ on $K'_n(X)$.

The group $K_0'(X)$ coincides with the Grothendieck group of coherent $\cO_X$-modules. The terms $E_{i,-i}^2=\CH_i(X)$ of the second page are the \emph{Chow groups} of classes of dimension $i$ algebraic cycles on $X$. The natural surjective homomorphism
\[
\varphi_i:\CH_i(X)\tto K'_0(X)_{(i/i-1)}:=K'_0(X)_{(i)}/K'_0(X)_{(i-1)}
\]
takes the class $[Z]$ of an integral closed subvariety $Z\subset X$ of dimension $i$ to the class of $\cO_Z$. The kernel of $\varphi_i$ is covered by the images of the differentials in the spectral sequence with target in $\CH_i(X)$.

The groups
\[
\CK_i(X):=D^2_{i+1,-i-1}=\Im(K_0(\cM_i(X))\to K_0(\cM_{i+1}(X)))
\]
are the \emph{connective $K_0$-groups} of $X$ (see \cite{Cai08}). These groups are related to the Chow groups via exact sequences
\[
\CK_{i-1}(X)\to \CK_i(X)\to \CH_i(X) \to 0
\]
In the present paper we study differentials in the spectral sequence with target in the Chow groups via the connective $K_0$-groups. In Sections \ref{endomod} and \ref{engovar} we introduce and study the notion of an endo-module associated with an algebraic variety that locates a part of the BGQ spectral sequence near the zero diagonal.

In Section \ref{adams} we introduce an approach based on the Adams operations of homological type on the Grothendieck group. Compatibility of the Adams operations with the differentials in the spectral sequence was proved in \cite[Corollary 5.5]{Levine97} with the help of heavy machinery of higher $K$-theory. We give an elementary proof of the compatibility with the differential coming to the zero diagonal of the spectral sequence. The Adams operations are applied in Section \ref{applications} to the study of the kernel of the homomorphism $\varphi_i$, and of the relations between the Grothendieck group and its graded group with respect to the topological filtration.

In Section \ref{equivar} we consider the endo-module arising form the equivariant analog of the BGQ spectral sequence. As an example we compute the connective $K_0$-groups of the classifying space of the special orthogonal group $\O^+_n$ with $n$ odd and as an application compute the differentials in the spectral sequence.\\

We use the following notation in the paper. We fix a base field $F$. A variety is a separated scheme of finite type over $F$. The residue field of a variety at a point $x$ is denoted by $F(x)$, and the function field of an integral variety $X$ by $F(X)$. The tangent bundle of a smooth variety $X$ is denoted by $\Tan_X$.

\section{Endo-modules}\label{endomod}
\begin{dfn}
Let $R$ be a commutative ring and $B_\bullet$ a $\Z$-graded $R$-module. An endomorphism of $B_\bullet$ of degree $1$ is (an infinite) sequence of $R$-module homomorphisms
\[
\dots \xra{\beta_{i-2}} B_{i-1} \xra{\beta_{i-1}} B_{i} \xra{\ \beta_{i}\ } B_{i+1} \xra{\beta_{i+1}}\dots
\]
We call the pair $(B_\bullet, \beta_\bullet)$ an \emph{endo-module} over $R$. If $\beta_\bullet$ is clear from the context, we simply write $B_\bullet$ for $(B_\bullet, \beta_\bullet)$.
\end{dfn}

For an endo-module $(B_\bullet, \beta_\bullet)$ set
\[
A_i=\Ker(\beta_i)\quad\text{and}\quad C_i=\Coker(\beta_{i-1}).
\]
We have exact sequences
\[
0\to A_{i}\xra{\alpha_i} B_{i} \xra{\beta_{i}} B_{i+1} \xra {\gamma_{i+1}} C_{i+1} \to 0
\]
and (an infinite) diagram of $R$-module homomorphisms:
\[
\xymatrix{
\dots  &  A_{i-1}  \ar[d]^{\alpha_{i-1}} &  A_{i} \ar[d]^{\alpha_i} & A_{i+1}  \ar[d]^{\alpha_{i+1}} & \dots  \\
\dots \ar[r]^{\beta_{i-2}} &  B_{i-1}  \ar[r]^{\beta_{i-1}} \ar[d]^{\gamma_{i-1}} &  B_{i} \ar[r]^{\beta_{i}} \ar[d]^{\gamma_{i}} & B_{i+1}  \ar[r]^{\beta_{i+1}} \ar[d]^{\gamma_{i+1}} & \dots \\
\dots &   C_{i-1}  & C_{i}  & C_{i+1} & \dots
}
\]

The compositions $\delta_i=\gamma_{i}\circ \alpha_i:A_i\to C_i$ are called the \emph{differentials}.

We define the \emph{derived} endo-module $(B^{(1)}_\bullet, \beta^{(1)}_\bullet)$ of $(B_\bullet, \beta_\bullet)$ by
\[
B^{(1)}_i=\Im(\beta_i)\subset B_{i+1}\quad\text{and}\quad \beta^{(1)}_i=\beta_{i+1}|_{B^{(1)}_i}.
\]

Then the \emph{derivatives} of $A_i$'s and $C_i$'s are:
\begin{align*}
 A^{(1)}_i:= & \Ker(\beta^{(1)}_{i})\simeq\Ker(\delta_{i+1}), \\
 C^{(1)}_i:= & \Coker(\beta^{(1)}_{i-1})\simeq \Coker(\delta_i).
\end{align*}

For an integer $s>0$ denote $A_\bullet^{(s)}$, $B_\bullet^{(s)}$ and $C_\bullet^{(s)}$ the iterated $s$th derivatives of $A_\bullet$, $B_\bullet$ and $C_\bullet$.

\begin{example}
Let $(D^1_{r,s},E^1_{r,s})$ be an exact couple of $R$-modules (see \cite[\S 5.9]{Weibel94}) such that $D^1_{r,s}=0$ if $r+s<0$. The exact sequences
\[
E^1_{i+1,-i}\to D^1_{i,-i}\to D^1_{i+1,-i-1}\to E^1_{i+1,-i-1}\to 0
\]
for all $i$ yield an endo-module $B_i=D_{i,-i}$ over $R$. The associated group $A_i$ coincides with the image of the first homomorphism in the exact sequence and $C_i=E^1_{i,-i}$. The differential $E^1_{i+1,-i}\to D^1_{i,-i}\to E^1_{i,-i}$ on the first page of the spectral sequence associated with the exact couple factors into the composition $E^1_{i,-i}\tto A_i\xra{\delta_i} C_i$. The derived endo-module of $B_\bullet$ arises the same way from the derived exact couple. It follows that the differential $\delta^{(s)}$ in the $s$th derivative $B^{(s)}_\bullet$ correspond to the differentials in the $(s+1)$th page of the spectral sequence.
\end{example}

For an endo-module $B_\bullet$ write $H=H(B_\bullet):=\colim B_i$. For every $i$, denote $H_{(i)}$ the image of the canonical homomorphism $B_i\to H$. We have a filtration
\[
\dots\subset H_{(i-1)}\subset H_{(i)} \subset H_{(i+1)}\subset\dots
\]
of $H$. We would like to compute the subsequent factor modules
\[
H_{(i/i-1)}:= H_{(i)}/H_{(i-1)}
\]
in terms of the $C_i$'s.

There is a canonical surjective homomorphism
\[
\varepsilon_i:C_i=\Coker(\beta_{i-1})\tto  H_{(i/i-1)}
\]
defined via the diagram $C_i \leftarrow\!\!\!\!\!\leftarrow B_i\to H_{(i)}$ and a commutative diagram
\[
\xymatrix{
B_i  \ar@{->>}[d]_{\gamma_i} \ar@{->>}[r] & H_{(i)}  \ar@{->>}[d]  \\
C_i \ar@{->>}[r]^-{\varepsilon_i}     &  H_{(i/i-1)}.
}
\]

We have $H^{(1)}:=H(B^{(1)}_\bullet)=H(B_\bullet)=H$ and $H^{(1)}_{(i)}=H_{(i)}$ for all $i$. The homomorphism $\varphi_i$ factors into the composition
\[
C_i\tto C^{(1)}_i\xra{\varepsilon^{(1)}_i} H^{(1)}_{(i/i-1)}=H_{(i/i-1)}.
\]
We call the $\varepsilon^{(1)}_i$ the \emph{derivative} of $\varepsilon_i$.

Iterating we factor $\varepsilon_i$ into the composition
\[
C_i\tto C^{(1)}_i\tto C^{(2)}_i\tto\cdots \tto C^{(s)}_i\xra{\varepsilon^{(s)}_i} H_{(i/i-1)}
\]
that yields an isomorphism
\[
\underset{s}{\colim} \ C_i^{(s)} \iso H_{(i/i-1)}
\]
for every $i$. Recall that $C^{(s)}_i=\Coker(A^{(s-1)}_{i}\xra{\delta_{i}} C^{(s-1)}_{i})$.

We would like to find conditions on $B_\bullet$ such that for every $i$ the iterated derivative $\varepsilon^{(s)}_i$ of sufficiently large order $s$ is an isomorphism.

\begin{dfn}
An endo-module $B_\bullet$ is called \emph{$d$-stable} for an integer $d$ if $A_i=0$ for all $i\geq d$. We say that $B_\bullet$ is \emph{stable} if $B_\bullet$ is $d$-stable for some $d$, $B_\bullet$ is \emph{degenerate} if $B_\bullet$ is $d$-stable for all $d$ (equivalently, the homomorphisms $B_i\to B_{i+1}$ are injective for all $i$, or all $A_i$ are zero) and $B_\bullet$ is \emph{bounded below} if $B_i=0$ for $i<\!\!<0$.
\end{dfn}
The following properties are straightforward.

\begin{lemma}\label{props}
Let $B_\bullet$ be an endo-module. Then
\begin{enumerate}
  \item If $B_\bullet$ is $d$-stable, then
\begin{enumerate}
\item The $s$th derivative $B^{(s)}_\bullet$ is $(d-s)$-stable,

\item $B^{(s)}_i=B_i$ and $C^{(s)}_i=C_i$ for $i\geq d$,

\item $\varepsilon^{(s)}_i:C^{(s)}_i\to H_{(i/i-1)}$ is an isomorphism if $i+s\geq d$.
\end{enumerate}

  \item If $B_\bullet$ is stable and bounded below, then $B_\bullet^{(s)}$ is degenerate for $s>\!\!>0$.
  \item If $B_\bullet$ is degenerate, then $\varepsilon_i$ is an isomorphism for all $i$. The converse holds if $B_\bullet$ is bounded below.
  \item If $B_\bullet$ is $d$-stable, bounded below and $C_i=0$ for all $i<d$, then $B_\bullet$ is degenerate.
\end{enumerate}
\end{lemma}

\section{The endo-module of a variety}
\label{engovar}

\subsection{The endo-module \texorpdfstring{$B_i(X)$}{Bi(X)}}

Let $X$ be a variety. We will denote by $K_0'(X)$ (resp.\ $K_0(X)$) the Grothendieck group of the category $\cM(X)$ of coherent (resp.\ the category of locally free coherent) $\cO_X$-modules. The class of an $\cO_X$-module $M$ in either of these groups will be denoted by $[M]$. The tensor product endows $K_0(X)$ with a ring structure, and $K_0'(X)$ with a $K_0(X)$-module structure.
We will denote the latter by $(a,b) \mapsto a \cdot b$, where $a\in K_0(X)$ and $b\in K_0'(X)$.

For an integer $i$ denote $\cM_i(X)$ the abelian category of coherent $\cO_X$-modules of support dimension at most $i$. Clearly, $\cM_i(X)=0$ if $i<0$ and $\cM_i(X)=\cM(X)$ if $i\geq d=\dim(X)$.

\begin{dfn}
We define an endo-module $(B_\bullet(X), \beta_\bullet)$ over $\Z$ associated with $X$ as follows. Set
\[
B_i(X)=K_0(\cM_i(X))
\]
and let $\beta_{i-1}:B_{i-1}(X)\to B_{i}(X)$ be the homomorphism induced by the inclusion of $\cM_{i-1}(X)$ into $\cM_{i}(X)$.
\end{dfn}

We have $B_i(X)=0$ if $i<0$ and $B_i(X)=K'_0(X)$ if $i\geq d$, so the endo-module $B_\bullet(X)$ is bounded below and $d$-stable. Also
\[
B_i(X)=\colim K'_0(Z),
\]
where the colimit is taken over all closed subvarieties $Z\subset X$ of dimension at most $i$ with respect to the push-forward homomorphisms $K'_0(Z_1)\to K'_0(Z_2)$ for closed subvarieties $Z_1\subset Z_2$. The group $H=\colim B_i(X)$ coincides with $K'_0(X)$ and $H_{(i)}$ with the $i$th term $K'_0(X)_{(i)}$ of the topological filtration on $K'_0(X)$.

The factor category $\cM_{i}(X)/\cM_{i-1}(X)$ is isomorphic to the direct sum over all points $x\in X_{(i)}$ of the categories $\cM(\Spec F(x))$ (see \cite[\S 7]{Quillen73}). The localization exact sequence \cite[\S7]{Quillen73} looks then as follows:
\[
C_{i}(X,1)\xra{\partial_i} B_{i-1}(X) \xra{\beta_{i-1}} B_{i}(X)\to C_{i}(X)\to 0,
\]
where
\[
C_i(X,1) = \Coprod_{x \in X_{(i)}} F(x)^{\times}\quad\text{and}\quad C_i(X)=\Coker(\beta_{i-1})=\Coprod_{x\in X_{(i)}}\Z
\]
is the group of algebraic cycles of dimension $i$. The groups $A_i(X)$ associated with the endo-module $B_\bullet(X)$
are given then by
\begin{equation}\label{formulafor}
A_i(X)=\Ker(\beta_{i})=\Im(\partial_{i+1}).
\end{equation}

If $f\colon Y \to X$ is a proper morphism, there are homomorphisms $f_* \colon B_i(Y) \to B_i(X)$. There are also homomorphisms $f_* \colon C_i(Y,1) \to C_i(X,1)$, defined by letting the homomorphism $F(y)^\times \to F(x)^\times$ be trivial unless $f(y) =x$, in which case it is given by the norm of the finite degree field extension $F(y)/F(x)$
(see \cite[\S 1.4]{Fulton98}). We have
\begin{equation}
\label{eq:partial_push}
\partial_i \circ f_* = f_* \circ \partial_i.
\end{equation}

If $f\colon Y \to X$ is a flat morphism of relative dimension $r$, there are homomorphisms $f^* \colon B_i(X) \to B_{i+r}(Y)$. There are also homomorphisms $f^* \colon C_i(X,1) \to C_{i+r}(Y,1)$, defined by letting the homomorphism $F(x)^\times \to F(y)^\times$ be trivial unless $f(y) =x$, in which case it is given by the inclusion $F(x) \subset F(y)$
(see \cite[\S 1.7]{Fulton98}). We have
\begin{equation}
\label{eq:partial_pull}
\partial_{i+r} \circ f^* = f^* \circ \partial_i.
\end{equation}

\subsection{Connective \texorpdfstring{$K$}{K}-groups}
\begin{dfn}
The derivatives $B_i(X)^{(1)}$ of $B_i(X)$ are the \emph{connective $K$-groups} $\CK_i(X)$ and $C_i(X)^{(1)}$ are the \emph{Chow groups} $\CH_i(X)$ of classes of cycles of dimension $i$ (see \cite{Cai08}).
\end{dfn}
We have the exact sequences
\[
\CK_{i-1}(X)\xra{\beta} \CK_i(X) \to \CH_i(X) \to 0,
\]
where $\beta=\beta_{i-1}^{(1)}$'s are called the \emph{Bott homomorphisms}.

We can view the graded group $\CK_\bullet(X)$ as a module over the polynomial ring $\Z[\beta]$. It follows from the definition that
\begin{equation}\label{twoiso}
\CK_\bullet(X)/\beta \CK_\bullet(X)\simeq \CH_\bullet(X)\quad\text{and}\quad \CK_\bullet(X)/(\beta-1)\CK_\bullet(X)\simeq K'_0(X).
\end{equation}

For every $i\geq 0$ the (surjective) homomorphism
\[
\varphi_i:=\varepsilon^{(1)}_i: \CH_i(X)\tto K'_0(X)_{(i/i-1)}
\]
takes the class $[Z]$ of an integral closed subvariety $Z\subset X$ of dimension $i$ to the class of $\cO_Z$. The relations between the groups $\CK_i(X)$, $\CH_i(X)$ and $K'_0(X)_{(i)}$ are given by a commutative diagram
\[
\xymatrix{
\CK_i(X)  \ar@{->>}[d]_{\gamma_i} \ar@{->>}[r] & K'_0(X)_{(i)}  \ar@{->>}[d]  \\
\CH_i(X) \ar@{->>}[r]^-{\varphi_i}     &  K'_0(X)_{(i/i-1)}.
}
\]

The goal is to study the homomorphisms $\varphi_i$. Recall that $C_i(X)^{(1)}=\CH_i(X)$ and the groups $C_i(X)^{(s)}$ are inductively defined via the exact sequences
\[
A_i(X)^{(s)}\xra{\delta^{(s)}_i} C_i(X)^{(s)} \to C_i(X)^{(s+1)} \to 0.
\]

\begin{proposition}\label{comp}
The homomorphism $\varphi_i$ factors as the composition
\[
\CH_i(X)=C_i(X)^{(1)}\tto C_i(X)^{(s)}\xra{\varepsilon^{(s)}_i} K'_0(X)_{(i/i-1)},
\]
where $\varepsilon^{(s)}_i$ is an isomorphism if $s\geq d-i$.
\end{proposition}

\begin{remark}
The groups $B_i(X)$ and $B_i(X)^{(1)}=\CK_i(X)$ (but not $B_i(X)^{(s)}$ with $s>1$), viewed as generalized homology theories, satisfy the localization property (see \cite[Definition 4.4.6]{LM07}). The derivatives $B_i(X)^{(s)}$ (but not $B_i(X)$) satisfy homotopy property (see \cite[Definition 5.1.3]{LM07}) if $s\geq 1$. Thus, the first derivative (the connective $K$-theory) is the only derivative that satisfies both localization and homotopy properties.
\end{remark}

\subsection{Generators for \texorpdfstring{$A_i(X)$}{Ai(X)}}

\begin{dfn}
Let $L$ be a line bundle (locally free coherent $\cO_X$-module of constant rank $1$) over a variety $X$, and $s\in H^0(X,L)$ a section. We denote by $\cZ(s)$ the closed subscheme of $X$ whose ideal is the image of $s^{\vee} \colon L^\vee \to \cO_X$, and by $\cD(s)$ its open complement. The section $s$ is called \emph{regular} if the morphism $s \colon \cO_X \to L$ (or equivalently $s^{\vee} \colon L^\vee \to \cO_X$) is injective. In this case, the immersion $\cZ(s) \to X$ is an effective Cartier divisor.
\end{dfn}

If $s$ is a regular section of a line bundle $L$ over $X$, the exact sequence of $\cO_X$-modules
\[
0 \to L^{\vee} \xra{s^\vee} \cO_X \to \cO_{\cZ(s)} \to 0
\]
shows that
\begin{equation}
\label{eq:OZs}
[\cO_{\cZ(s)}] = [\cO_X] - [L^{\vee}] \in K_0'(X).
\end{equation}

\begin{notation}
\label{not:0_infty}
Let us write $\Pu = \Proj(F[x,y])$, and view $x$ and $y$ as sections of $\cO(1)$. We also view $x/y$ (resp.\ $y/x$) as a regular function on $\cD(y)$ (resp.\ $\cD(x)$). Mapping $u$ to that function induces an isomorphism between $\Au = \Spec(F[u])$ and $\cD(y)$ (resp.\ $\cD(x)$).
\end{notation}

\begin{lemma}
\label{lemm:partial_Pu}
We have $\partial_0(x/y) = [\cO_{\cZ(x)}] - [\cO_{\cZ(y)}]$ in $B_0(\Pu)$.
\end{lemma}
\begin{proof}
Restricting to the open subschemes $\cD(x), \cD(y)$ induces an injective map $B_0(\Pu) \to B_0(\cD(x)) \oplus B_0(\cD(y)$. Thus we are reduced to proving that $\partial_0(u) = [\cO_{\cZ(u)}] \in B_0(\Au)$ under the identification $\Au = \Spec(F[u])$. This is done, e.g.\ in \cite[\S7, Lemma 5.1]{Quillen73}.
\end{proof}

\begin{proposition}
\label{prop:ker}
Let $X$ be a variety and $i\in \Z$. The subgroup $A_i(X) \subset B_i(X)$ is generated by the elements $f_*([\cO_{\cZ(s_1)}] - [\cO_{\cZ(s_2)}])$, where
\begin{itemize}
\item $f\colon Y \to X$ is a proper morphism,
\item $Y$ is quasi-projective and integral of dimension $i+1$,
\item $s_1,s_2$ are regular sections of a common line bundle over $Y$.
\end{itemize}
\end{proposition}
\begin{proof}
Let $S_i(X) \subset B_i(X)$ be the subgroup generated by the elements $f_*([\cO_{\cZ(s_1)}] - [\cO_{\cZ(s_2)}])$ as in the statement. For such $s_1,s_2$, we have $[\cO_{\cZ(s_1)}] = [\cO_{\cZ(s_2)}] \in K_0'(Y) = B_{i+1}(Y)$ by \eqref{eq:OZs}, and thus $S_i(X) \subset A_i(X)$.

It follows from \eqref{formulafor} and \eqref{eq:partial_push} that the subgroup $A_i(X) \subset B_i(X)$ is generated by the push-forwards of elements $\partial_i(a) \in B_i(Z)$, where $a \in F(Z)^{\times}$ with $Z \subset X$ an integral closed subscheme of dimension $i+1$. Let $U$ be a dense open subscheme of $Z$ such that $a \in H^0(U,\cO_U) \subset F(Z)$. Mapping $u$ to $a$ induces a morphism $U \to \Spec(F[u]) = \Au$. Composing with the morphism $\Au \simeq \cD(y) \subset \Pu$ (using Notation \ref{not:0_infty}), we obtain a morphism $U \to \Pu$. We denote by $S$ the closure in $\Pu$ of the image of the latter morphism, endowed with the reduced scheme structure. Consider the graph of the morphism $U \to S$ as a closed subset of $U \times S$, and let $Y'$ be its closure in $Z \times S$, endowed with the reduced scheme structure. By Chow's lemma \cite[(5.6.1)]{ega-2} we may find a proper birational morphism $Y \to Y'$, where $Y$ is quasi-projective and integral. Then we have morphisms $Z \xleftarrow{f} Y \xrightarrow{g} S$. The morphism $f$ is proper and birational, hence $a$ admits a pre-image $b$ under the isomorphism $f_*\colon F(Y)^\times = C_{i+1}(Y,1) \to C_{i+1}(Z,1) = F(Z)^\times$.  The morphism $g$ is dominant, hence flat by \cite[III 9.7]{Har-Alg-77}.

If $\dim S=0$ (i.e.\ $a$ is constant), then $b=g^*c$ for some $c \in F(S)^{\times}$, and the morphism $g$ has relative dimension $i+1$, so that, by \eqref{eq:partial_push} and \eqref{eq:partial_pull}
\[
\partial_i(a) = f_* \circ \partial_i(b) = f_* \circ g^* \circ \partial_{-1}(c) \subset f_* \circ g^*B_{-1}(S)=0.
\]

Otherwise $S=\Pu$, and $g$ has relative dimension $i$. Using Notation \ref{not:0_infty}, we have $b = g^*(x/y)$. By Lemma \ref{lemm:partial_Pu} and \eqref{eq:partial_pull}, we have in $B_i(Y)$
\[
\partial_i(b) =  g^* \circ \partial_0(x/y) = g^*([\cO_{\cZ(x)}] - [\cO_{\cZ(y)}]) = [\cO_{g^{-1}\cZ(x)}] - [\cO_{g^{-1}\cZ(y)}].
\]
The flatness of $g$ implies that the sections $s_1 := g^*x$ and $s_2 := g^*y$ of $g^*\cO(1)$ are regular, and satisfy $\cZ(s_1) = g^{-1}\cZ(x)$ and $ \cZ(s_2) = g^{-1}\cZ(y)$. Using \eqref{eq:partial_push}, we deduce that $\partial_i(a) = f_* \circ \partial_i(b)= f_*([\cO_{\cZ(s_1)}] -[\cO_{\cZ(s_2)}])$ in $B_i(X)$, and we have proved that $A_i(X) \subset S_i(X)$.
\end{proof}

\section{Homological Adams operations}\label{adams}
\subsection{\texorpdfstring{$K$}{K}-theory with supports}

\begin{dfn}
\label{def:K_support}
Let $X$ be a variety and $Y \subset X$ a closed subscheme. We consider the category of chain complexes of locally free coherent $\cO_X$-modules
\[
E_\bullet = \cdots \to E_n \to E_{n-1} \to \cdots
\]
satisfying $E_i=0$ when $i<0$ or $i>\!\!>0$. The full subcategory consisting of those complexes whose homology is supported on $Y$ will be denoted by  $\cC^Y(X)$. We define the group $K_0^Y(X)$ as the free abelian group generated by the elements $[E_\bullet]$, where $E_\bullet$ runs over the isomorphism classes of objects in $\cC^Y(X)$, modulo the following relations:
\begin{itemize}
\item If $0 \to E_\bullet' \to E_\bullet \to E_\bullet'' \to 0$ is an exact sequence of complexes in $\cC^Y(X)$, then $[E_\bullet] = [E'_\bullet] + [E''_\bullet]$ in $K_0^Y(X)$.

\item If $E_\bullet \to E'_\bullet$ is a quasi-isomorphism in $\cC^Y(X)$, then $[E_\bullet] = [E'_\bullet]$ in $K_0^Y(X)$.
\end{itemize}
\end{dfn}

When $P$ is a locally free coherent $\cO_X$-module and $i\in \N$, we will denote the complex
\begin{equation}
\label{eq:[i]}
\cdots \to 0 \to P \to 0 \to \cdots
\end{equation}
concentrated in degree $i$, by $P[i] \in \cC^X(X)$. We will write $1 := \cO_X[0] \in \cC^X(X)$.\\

Let $X$ be a variety and $Y \subset X$ a closed subscheme. There is a bilinear map
\[
K_0(X) \times K_0^Y(X) \to K_0^Y(X); \quad (a ,\beta) \mapsto a \cdot \beta,
\]
such that for any locally free coherent $\cO_X$-modules $P$ and $E_\bullet \in \cC^Y(X)$ we have
\[
[P] \cdot [E_{\bullet}] = [P \otimes_{\cO_X} E_{\bullet}] \in K_0^Y(X).
\]
If $Z \subset X$ is another closed subscheme, there is a bilinear map
\[
K_0^Y(X) \times K_0'(Z) \to K_0'(Y \cap Z); \quad (\alpha,b) \mapsto \alpha \cap b,
\]
such that for any $E_\bullet \in \cC^Y(X)$ and $M \in \cM(Z)$ we have
\[
[E_\bullet] \cap [M] = \sum_{i\in\N} (-1)^i[H^i(E_\bullet \otimes_{\cO_X} M)] \in K_0'(Y \cap Z).
\]
If $f \colon X' \to X$ is a morphism, there is a pullback homomorphism
\[
f^* \colon K_0^Y(X) \to K_0^{f^{-1}Y}(X').
\]

We will need the following basic compatibilities, which may be verified at the level of modules (before applying the functor $K_0'$).

\begin{lemma}
\label{lemm:compat}
Let $X$ be a variety and $Y,Z$ closed subschemes of $X$. Let $\alpha \in K_0^Y(X)$ and $b \in K_0'(Z)$. Denote by $i\colon Z \to X$ the closed immersion.
\begin{enumerate}[label=(\alph*), ref=\ref{lemm:compat}.\alph*]
\item
\label{compat:cap_commutes}
Let $Y'$ be a closed subscheme of $X$ and $\alpha' \in K_0^{Y'}(X)$. Then
\[
\alpha \cap (\alpha' \cap b) = \alpha' \cap (\alpha \cap b) \in K_0'(Y \cap Y'\cap Z).
\]

\item
\label{compat:cdot_cap}
Denote by $f \colon Y\cap Z \to X$ the closed immersion. For any $e \in K_0(X)$,
\[
(e \cdot \alpha) \cap b = \alpha\cap ( (i^*e) \cdot b) = (f^*e) \cdot (\alpha \cap b) \in K_0'(Y \cap Z).
\]

\item
\label{compat:cap_pullback}
Denote by $g \colon Y \to X$ the closed immersion. Then
\[
\alpha \cap b = (g^*\alpha) \cap b \in K_0'(Y \cap Z).
\]

\item
\label{compat:cap_pushforward_2}
If $Y \subset Z$, then
\[
\alpha \cap b = \alpha \cap i_*b \in K_0'(Y).
\]

\item
\label{compat:cap_pushforward_1}
Assume that $Y \subset Z$, and denote by $j\colon Y \to Z$ the closed immersion. Then
\[
j_*(\alpha \cap b) = \tilde{\alpha} \cap b \in K_0'(Z),
\]
where $\tilde{\alpha}$ is the image of $\alpha$ under the ``forgetful'' map $K_0^Y(X) \to K_0^X(X)$.
\end{enumerate}
\end{lemma}

\begin{lemma}[{\cite[Lemma 1.9]{GS-87}}]
\label{lemm:regular}
Let $X$ be a regular variety and $Y \subset X$ a closed subscheme. Then the following map is an isomorphism:
\[
K_0^Y(X) \to K_0'(Y) \quad ; \quad \alpha \mapsto \alpha \cap [\cO_X].
\]
\end{lemma}

\begin{dfn}
Let $L$ be a line bundle over $X$, and $s\in H^0(X,L)$ a section. We will denote by $K(s)$ the complex of locally free coherent $\cO_X$-modules
\[
\cdots \to 0 \to L^\vee \xrightarrow{s^{\vee}} \cO_X \to 0 \to \cdots
\]
concentrated in degrees $1,0$.
\end{dfn}

The homology of $K(s)$ is supported on $\cZ(s)$, so that we have a class $[K(s)] \in K_0^{\cZ(s)}(X)$. If the section $s$ is regular, then
\begin{equation}
\label{eq:K(s)_Z(s)}
[K(s)] \cap [\cO_X] = [\cO_{\cZ(s)}] \in K_0'(\cZ(s)).
\end{equation}

\begin{lemma}
\label{lemm:forget_supports}
Let $L$ be a line bundle over a variety $X$. Then the image of $[K(s)]$ in $K_0^X(X)$ does not depend on the choice of the section $s \in H^0(X,L)$.
\end{lemma}
\begin{proof}
The commutative diagram with exact rows
\[ \xymatrix{
0 \ar[r] &0\ar[r] \ar[d] & L^{\vee} \ar[d]^{s^\vee} \ar[r]^{\id} & L^{\vee} \ar[r] \ar[d] & 0 \\
0 \ar[r] & \cO_X \ar[r]^{\id} & \cO_X \ar[r] & 0 \ar[r] & 0
}\]
shows that the image of $[K(s)]$ in $K_0^X(X)$ is $1 + [L^\vee[1]]$ (see \eqref{eq:[i]}). This element is visibly independent of $s$.
\end{proof}

\subsection{Bott's class}
From now on we fix a nonzero integer $k$.

\begin{lemma}
\label{lemm:nilpotent}
Let $L$ be a line bundle over a quasi-projective variety $X$. Then $1-[L] \in K_0(X)$ is nilpotent.
\end{lemma}
\begin{proof}
We may write $L=A \otimes B^{\vee}$ where $A,B$ are line bundles over $X$ such that $A^\vee, B^\vee$ are generated by their global sections. If $1-[A]$ and $1-[B]$ are nilpotent, then so is
\[
1-[L] = (1-[A]) - [B^\vee](1-[B]) + [B^\vee] (1-[A])(1-[B]).
\]
Thus we are reduced to assuming that $L^\vee$ is generated by its global sections. Pulling back along the associated morphism $X \to \Pp^n$, we reduce to $X=\Pp^n$ and $L=\cO_{\Pp^n}(-1)$. We prove by induction on $n$ that $(1-[L])^{n+1}=0 \in K_0(X)$. There is a regular section $s$ of $L^\vee$ such that $\cZ(s) = \Pp^{n-1}$ and $L|_{\cZ(s)} = \cO_{\Pp^{n-1}}(-1)$. Let $i\colon \cZ(s) \to X$ be the immersion. By \eqref{eq:OZs} and the projection formula, we have in $K_0'(X)$
\[
(1-[L])^{n+1} \cdot [\cO_X] = (1-[L])^n \cdot i_*[\cO_{\cZ(s)}] = i_*((1-[L|_{\cZ(s)}])^n \cdot [\cO_{\cZ(s)}]).
\]
That element vanishes by induction. Since the natural homomorphism $K_0(X) \to K_0'(X)$ is an isomorphism \cite[\S7.1]{Quillen73}, the claim follows.
\end{proof}

\begin{dfn}
Consider the power series
\[
\tau^k(c) = \frac{1 - (1-c)^k}{c} \in \Z[[c]].
\]
By Lemma \ref{lemm:nilpotent} and the splitting principle, there is a unique way to assign to each vector bundle $E$ over a variety $X$ an element $\theta^k(E) \in K_0(X)$ so that:
\begin{itemize}
\item If $L$ is a line bundle, then $\theta^k(L) = \tau^k(1-[L])$.

\item If $0 \to E' \to E \to E'' \to 0$ is an exact sequence of vector bundles, then $\theta^k(E) = \theta^k(E') \theta^k(E'')$.

\item If $f\colon Y \to X$ is a morphism and $E$ a vector bundle over $X$, then $f^*\theta^k(E) = \theta^k(f^*E)$.
\end{itemize}
\end{dfn}

The power series $\tau^k(c) - k$ is divisible by $c$ in $\Z[[c]]$, and thus $\tau^k(c)$ admits a multiplicative inverse in $\Z[1/k][[c]]$. We deduce, using Lemma \ref{lemm:nilpotent} and the splitting principle, that $\theta^k(E)$ is invertible in $K_0(X)[1/k]$ for any vector bundle $E$ over a variety $X$. Thus for every variety $X$, the association $E \mapsto \theta^k(E)$ extends uniquely to a map
\[
\theta^k\colon K_0(X) \to K_0(X)[1/k]
\]
satisfying $\theta^k(a - b) = \theta^k(a) \theta^k(b)^{-1}$ for any $a,b\in K_0(X)$.\\

For any variety $X$, we have $\theta^k(1) = \tau^k(0)=k$, and therefore
\begin{equation}
\label{eq:theta_trivial}
\theta^k(n) = k^n \text{ for any } n \in \Z \subset K_0(X).
\end{equation}

\subsection{Adams operations}
The classical Adams operation $\psi^k \colon K_0(-) \to K_0(-)$ is defined using the splitting principle by the following conditions:
\begin{itemize}
\item If $L$ is a line bundle, then $\psi^k [L] = [L^{\otimes k}]$.

\item For any $a,b \in K_0(X)$, we have $\psi^k(a-b) = \psi^k(a) - \psi^k(b)$.

\item If $f\colon Y \to X$ is a morphism, then $\psi^k \circ f^* = f^* \circ \psi^k$.
\end{itemize}

This construction may be refined to obtain an operation on the $K$-theory with supports:
\begin{dfn}
Let $X$ be a regular variety and $Y\subset X$ a closed subscheme. Then the group $K_0^Y(X)$ defined in \eqref{def:K_support} coincides with the one considered in \cite{Soule85}, as they are both canonically isomorphic to $K_0'(Y)$. Thus the construction of \cite{Soule85} yields an Adams operation $\psi^k\colon K_0^Y(X) \to K_0^Y(X)$.
\end{dfn}

The following properties follow from the construction given in \cite{Soule85}.

\begin{lemma}
\label{lemm:psi_support}
Let $X$ be a regular variety and $Y \subset X$ a closed subscheme.
\begin{enumerate}[label=(\alph*), ref=\ref{lemm:psi_support}.\alph*]
\item
\label{lemm:psi_support:pullback}
If $f\colon X' \to X$ is a morphism and $X'$ is regular, then
\[
f^* \circ \psi^k = \psi^k \circ f^* \colon K_0^Y(X) \to K_0^{f^{-1}Y}(X').
\]

\item
\label{lemm:psi_support:compo}
If $k' \in \Z-\{0\}$, then $\psi^k \circ \psi^{k'} = \psi^{kk'}$.

\item
\label{lemm:psi_support:cdot}
For any $a \in K_0(X)$ and $\beta \in K_0^Y(X)$, we have $\psi^k(a \cdot \beta) = (\psi^ka) \cdot (\psi^k\beta)$.

\item
\label{lemm:psi_support:1}
We have $\psi^k1=1$ in $K_0^X(X)$ (see \eqref{eq:[i]}).
\end{enumerate}
\end{lemma}

\begin{dfn}
\label{def:homological_Adams}
Any quasi-projective variety $X$ may be embedded as a closed subscheme of a smooth quasi-projective variety $W$. By Lemma \ref{lemm:regular}, there is a unique homomorphism
\[
\psi_k \colon K_0'(X)[1/k] \to K_0'(X)[1/k],
\]
called the $k$th \emph{Adams operation of homological type}, such that for any $\alpha \in K_0^X(W)$
\[
\psi_k(\alpha \cap [\cO_W]) = (\psi^k\alpha) \cap (\theta^k(-\cTan_W) \cdot [\cO_W]).
\]
It follows from the Adams Riemann--Roch theorem without denominators that this operation is independent of the choice of $W$, and that it commutes with proper push-forward homomorphisms (see \cite[Th\'eor\`eme 7]{Soule85}).
\end{dfn}

We now explain how to remove the assumption of quasi-projectivity. An \emph{envelope} is a proper morphism $Y \to X$ such that for each integral closed subscheme $Z \subset X$, there is an integral closed subscheme $W\subset Y$ such that the induced morphism $W \to Z$ is birational. Any base change of an envelope is an envelope, and the composition of two envelopes is an envelope \cite[Lemma 18.3 (2) (3)]{Fulton98}.

\begin{lemma}
\label{lemm:K0_cosheaf}
Let $f\colon Y \to X$ be an envelope. Denote by $p_1,p_2 \colon Y \times_X Y \to Y$ the two projections. Then the following sequence is exact
\[
K_0'(Y \times_X Y) \xrightarrow{(p_1)_* - (p_2)_*} K_0'(Y) \xrightarrow{f_*} K_0'(X) \to 0.
\]
\end{lemma}
\begin{proof}
The sequence is clearly a complex. We proceed by noetherian induction on $X$. Since push-forward homomorphisms along nilimmersions are bijective \cite[\S7, Proposition 3.1]{Quillen73}, we may assume that $X$ is reduced. Assuming that $X \neq \varnothing$, we may find a closed subscheme $X' \subsetneq X$ whose open complement $U$ is such that $f|_U\colon V:=f^{-1}U \to U$ admits a section $s \colon U \to V$ (letting $X_1,\dots,X_n$ be the irreducible components of $X$, we find $Y_1\subset Y$ birationally dominating $X_1$; then $Y_1 \to X_1$ restricts to an isomorphism over a nonempty open subscheme $U_1$ of $X_1$, and we set $U=U_1 \cap (X - (X_2 \cup \cdots \cup X_n))$). Let $Y'=f^{-1}(X')$, and consider the commutative diagram with exact rows \cite[\S7.3]{Quillen73}
\[ \xymatrix{
K_1'(V)\ar[r] \ar[d]^{(f|_U)_*} & K_0'(Y') \ar[r] \ar[d]^{(f|_{X'})_*} & K_0'(Y) \ar[d]^{f_*} \ar[r]& K_0'(V) \ar[d]^{(f|_U)_*} \ar[r] &0\\
K_1'(U)\ar[r] & K_0'(X') \ar[r] & K_0'(X) \ar[r] & K_0'(U)\ar[r] &0
}\]
Each homomorphism $(f|_U)_*$ is surjective, since it admits a section $s_*$. The homomorphism $(f|_{X'})_*$ is surjective by induction, and a diagram chase shows that $f_*$ is surjective.

Let now $a \in K_0'(Y)$ be such that $f_*a=0$ in $K_0'(X)$. Let $b_U \in K_0'(V \times_U V)$ be the image of $a|_V \in K_0'(V)$ under the push-forward homomorphism along $(\id_V,s \circ f|_U) \colon V \to V \times_U V$, and let $b \in K_0'(Y \times_X Y)$ be a pre-image of $b_U$. Then $(p_1)_*(b)|_V = a|_V$ and $(p_2)_*(b)|_V =0$. Thus $a - ((p_1)_* - (p_2)_*)(b) \in K_0'(Y)$ is the image of an element of $c \in K_0'(Y')$. Chasing the above diagram, we see that $c$ may be modified to satisfy additionally $(f|_{X'})_*(c) =0$. By induction $c$ is the image of an element of $K_0'(Y' \times_{X'} Y')$, whose push-forward $d \in K_0'(Y \times_X Y)$ satisfies $a = ((p_1)_* - (p_2)_*)(b+d)$. This concludes the proof.
\end{proof}

Since any variety $X$ admits an envelope $Y \to X$ where $Y$ is quasi-projective (see \cite[Lemma 18.3 (3)]{Fulton98}), combining Lemma \ref{lemm:K0_cosheaf} with \cite[Proposition~5.2]{2nd} yields:

\begin{proposition}
\label{prop:non_quasiproj}
There is a unique way to define an operation $\psi_k \colon K_0'(X)[1/k] \to K_0'(X)[1/k]$ for each variety $X$, compatibly with proper push-forward homomorphisms and agreeing with Definition \ref{def:homological_Adams} when $X$ is quasi-projective.
\end{proposition}

Using \eqref{lemm:psi_support:pullback} and the surjectivity of push-forward homomorphisms along envelopes, we see that the Adams operation $\psi_k$ commutes with the restriction to any open subscheme.\\

Let $k' \in \Z-\{0\}$ and let $X$ be a quasi-projective variety. Note that for any $a\in K_0(X)$
\begin{equation}
\label{eq:theta_kk'}
\theta^{k}(a) \cdot (\psi^k \circ \theta^{k'}(a)) = \theta^{kk'}(a) \in K_0(X)[1/kk'];
\end{equation}
this is immediate when $a$ is the class of a line bundle, and follows in general from the splitting principle. Combining \eqref{eq:theta_kk'} with \eqref{lemm:psi_support:compo}, \eqref{lemm:psi_support:cdot} and \eqref{compat:cdot_cap}, we deduce that
\begin{equation}
\label{eq:psi_kk'}
\psi_k \circ \psi_{k'} = \psi_{kk'} \colon K_0'(X)[1/kk'] \to K_0'(X)[1/kk'].
\end{equation}
By Proposition \ref{prop:non_quasiproj}, this formula remains valid when $X$ is an arbitrary variety.

\begin{dfn}
Let $X$ be a variety. Assume that there is a smooth variety $W$, and a regular closed immersion $i\colon X \to W$, with normal bundle $N$. The element
\[
\Tan_X := [\Tan_W|_X] -[N]\in K_0(X),
\]
does not depend on the choice of $W$ and $i$, and is called the \emph{virtual tangent bundle} of $X$ (see \cite[B.7.6]{Fulton98}).
\end{dfn}

\begin{lemma}
\label{lemm:psi_OX}
Let $X$ be a regular quasi-projective variety. Then
\[
\psi_k[\cO_X] = \theta^k(-\cTan_X)\cdot [\cO_X] \in K_0'(X)[1/k].
\]
\end{lemma}
\begin{proof}
Let $i\colon X \to W$ be a closed immersion, where $W$ is smooth and quasi-projective. Then $i$ is a regular closed immersion, let $N$ be its normal bundle. The Gysin homomorphism $i_* \colon K_0^X(X) \to K_0^X(W)$ is by definition the unique map compatible with the isomorphisms $K_0^X(X) \to K_0'(X)$ and $K_0^X(W) \to K_0'(X)$ of Lemma \ref{lemm:regular}. Then $(i_*1) \cap [\cO_W] = [\cO_X]$ in $K_0'(X)$ (see \eqref{eq:[i]}). By \eqref{lemm:psi_support:1} and the Adams Riemann--Roch theorem (see \cite[Th\'eor\`eme 3]{Soule85}, where $N$ should be replaced by $N^\vee$), we have in $K_0'(X)[1/k]$
\[
\psi_k[\cO_X] = (\psi^k \circ i_*1)\cap (\theta^k(-\cTan_W)\cdot [\cO_W]) = (\theta^k(N^\vee)\cdot (i_*1)) \cap (\theta^k(-\cTan_W)\cdot [\cO_W])),
\]
and the statement follows from \eqref{compat:cdot_cap}.
\end{proof}

\begin{lemma}
\label{lemm:psi_generic}
Let $X$ be an integral variety of dimension $d$. Then there is a nonempty open subscheme $U$ of $X$ such that $\psi_k[\cO_U] = k^{-d}[\cO_U]$ in $K_0'(U)[1/k]$.
\end{lemma}
\begin{proof}
Let $U$ be a quasi-projective regular nonempty open subscheme of $X$. The virtual tangent bundle $\Tan_U \in K_0(U)$ may be written as $[E] -[F]$, where $E,F$ are vector bundles over $U$. Shrinking $U$, we may assume that $E$ and $F$ are trivial, so that $\cTan_U =d \in K_0(U)$, and the statement follows from Lemma \ref{lemm:psi_OX} and \eqref{eq:theta_trivial}.
\end{proof}

\begin{proposition}\label{firstaction}
Let $X$ be a variety and $i\in \Z$. The operation $\psi_k$ acts on $C_i(X)[1/k]$ via multiplication by $k^{-i}$.
\end{proposition}
\begin{proof}
We may assume that $X$ is integral of dimension $i$. Then $C_i(X)$ is the free abelian group generated by the image of $[\cO_X] \in B_i(X) = K_0'(X)$, and the proposition follows from Lemma \ref{lemm:psi_generic}.
\end{proof}

\subsection{Adams operations on divisor classes}
\begin{lemma}
\label{lemm:extend_section}
Let $L$ be a line bundle over a quasi-projective variety $X$. Let $s \in H^0(X,L)$. Then we may find
\begin{itemize}
\item a closed immersion $X \to W$ where $W$ is smooth and quasi-projective,
\item a line bundle $M$ over $W$ such that $M|_X=L$,
\item a regular section $t \in H^0(W,M)$ such that $t|_X=s$ and $\cZ(t)$ is smooth.
\end{itemize}
\end{lemma}
\begin{proof}
By \cite[Lemma~18.2]{Fulton98}, we may find a smooth quasi-projective variety $V$ containing $X$ as a closed subscheme, and a line bundle $W \to V$ such that $W|_X=L$. Let $M=W \times_V W$, and view $M$ as a line bundle over $W$ via the first projection. The diagonal $W \to W \times_V W$ may be considered as a regular section $t$ of $M$ whose vanishing locus is $V$ (embedded in $W$ as the zero-section). We view $X$ as a closed subscheme of $W$ using the composite
\[
X \xrightarrow{s} L = W|_X \to W,
\]
where the last morphism is the base change of the immersion $X \to V$. The statements are then easily verified.
\end{proof}

\begin{lemma}
\label{lemm:psi_div}
Let $L$ be a line bundle over a quasi-projective variety $X$, and $s$ a regular section of $L$. Set $Y=\cZ(s)$. Then we have in $K_0'(Y)[1/k]$
\[
\psi_k [\cO_Y] = [K(s)] \cap (\theta^k(L^\vee) \cdot \psi_k[\cO_X]).
\]
\end{lemma}
\begin{proof}
Let us apply Lemma \ref{lemm:extend_section} and use its notation. By Lemma \ref{lemm:regular}, there is an element $\alpha \in K_0^X(W)$ such that
\begin{equation}
\label{eq:alpha}
\alpha \cap [\cO_W] = [\cO_X] \in K_0'(X).
\end{equation}
Let $V = \cZ(t)$ and $j\colon V \to W$ be the closed immersion. We have in $K_0'(Y)$
\begin{align*}
[\cO_Y]
&= [K(s)] \cap [\cO_X]  && \text{by \eqref{eq:K(s)_Z(s)}} \\
&= [K(t)] \cap (\alpha \cap [\cO_W]) && \text{by \eqref{eq:alpha} and \eqref{compat:cap_pullback}}\\
&=  \alpha \cap ([K(t)] \cap [\cO_W])&&\text{by \eqref{compat:cap_commutes}} \\
&= \alpha \cap [\cO_V]  && \text{by \eqref{eq:K(s)_Z(s)}} \\
&= (j^* \alpha) \cap [\cO_V] && \text{by \eqref{compat:cap_pullback}.}
\end{align*}
Since $[\Tan_V] = [\Tan_W|_V] - [M|_V]$ in $K_0(V)$, we have in $K_0'(Y)[1/k]$
\begin{align*}
\psi_k[\cO_Y]
&= \psi_k(j^*\alpha \cap [\cO_V]) = \psi^k(j^*\alpha) \cap (\theta^k(-\cTan_V) \cdot [\cO_V])\\
&= (\psi^k \alpha) \cap (\theta^k(-\cTan_V) \cdot [\cO_V])&& \text{by \eqref{compat:cap_pullback}, \eqref{lemm:psi_support:pullback}}\\
&= (\psi^k \alpha) \cap (\theta^k(M^\vee|_V) \theta^k(-\cTan_W|_V) \cdot ([K(t)] \cap [\cO_W])) && \text{by \eqref{eq:K(s)_Z(s)}}\\
&= [K(t)] \cap (\theta^k(L^\vee) \cdot ((\psi^k\alpha) \cap (\theta^k(-\cTan_W) \cdot [\cO_W])))&&\text{by \eqref{compat:cap_commutes}, \eqref{compat:cdot_cap}}\\
&= [K(s)] \cap (\theta^k(L^\vee) \cdot \psi_k[\cO_X]) && \text{by \eqref{compat:cap_pullback}, \eqref{eq:alpha}.}\qedhere
\end{align*}
\end{proof}

\begin{proposition}
\label{prop:2_sections}
Let $X$ be an integral quasi-projective variety of dimension $d$. Let $L$ be a line bundle over $X$, and $s_1,s_2$ regular sections of $L$. Then we may find a closed subscheme $Z \subsetneq X$ containing $\cZ(s_1)$ and $\cZ(s_2)$ as closed subschemes, and such that
\[
\psi_k([\cO_{\cZ(s_1)}] - [\cO_{\cZ(s_2)}]) = k^{1-d}([\cO_{\cZ(s_1)}] - [\cO_{\cZ(s_2)}]) \in K_0'(Z)[1/k].
\]
\end{proposition}
\begin{proof}
Since the sections $s_1,s_2$ are regular, we may find a nonempty open subscheme $U$ of $X$ which does not meet $\cZ(s_1) \cup \cZ(s_2)$. Then $L|_U$ is trivial. Shrinking $U$, we may assume that $\psi_k[\cO_U] = k^{-d} [\cO_U]$ in $K_0'(U)[1/k]$ by Lemma \ref{lemm:psi_generic}. Let $Z'$ be the reduced closed complement of $U$ in $X$. The intersection of the ideal sheaves of $Z', \cZ(s_1), \cZ(s_2)$ in $\cO_X$ defines a closed subscheme $Z \subset X$ whose open complement is $U$, and we have closed immersions $j_n \colon \cZ(s_n) \to Z$ for $n\in \{1,2\}$. Since $\theta^k(L^\vee|_U) = k$ by \eqref{eq:theta_trivial}, we have
\[
\theta^k(L^\vee|_U) \cdot \psi_k[\cO_U] =  k^{1-d} [\cO_U] \in K_0'(U)[1/k].
\]
It follows from the localization sequence \cite[\S7, Proposition 3.2]{Quillen73} that
\begin{equation}
\label{eq:theta_L_psi_OX}
\theta^k(L^\vee) \cdot \psi_k[\cO_X]= k^{1-d} [\cO_X] +i_*z\in K_0'(X)[1/k]
\end{equation}
where $z \in K_0'(Z)[1/k]$, and $i \colon Z \to X$ is the closed immersion. By Lemma \ref{lemm:forget_supports}, the image $\sigma \in K_0^X(X)$ of $[K(s_n)] \in K_0^{\cZ(s_n)}(X)$ does not depend on $n \in \{1,2\}$. For such $n$, we have in $K_0'(Z)[1/k]$
\begin{align*}
\psi_k \circ (j_n)_*[\cO_{\cZ(s_n)}] &=(j_n)_* \circ \psi_k[\cO_{\cZ(s_n)}]\\
&= (j_n)_*([K(s_n)] \cap (\theta^k(L^\vee) \cdot \psi_k[\cO_X])) && \text{by Lemma \ref{lemm:psi_div}} \\
&= k^{1-d} (j_n)_*[\cO_{\cZ(s_n)}] + (j_n)_*([K(s_n)] \cap i_*z) && \text{by \eqref{eq:K(s)_Z(s)}, \eqref{eq:theta_L_psi_OX}}\\
&= k^{1-d} (j_n)_*[\cO_{\cZ(s_n)}] + \sigma \cap z&& \text{by \eqref{compat:cap_pushforward_2}, \eqref{compat:cap_pushforward_1}}.
\end{align*}
The statement follows.
\end{proof}

Combining Propositions \ref{prop:2_sections} and \ref{prop:ker}, we obtain:
\begin{proposition}\label{secondaction}
Let $X$ be a variety and $i \in \Z$. The operation $\psi_k$ acts on $A_i(X)[1/k]$ via multiplication by $k^{-i}$.
\end{proposition}

\section{Applications of the Adams operations}
\label{applications}

\subsection{Inverting small primes}
For every nonzero integer $k$, the homological operation $\psi_k$ on the groups $K'_0(Z)[1/k]$ for all closed subschemes $Z\subset X$ yields an operation (still denoted $\psi_k$) on $B_i(X)[1/k]=\Colim_{\dim Z\leq i}K'_0(Z)[1/k]$ that commutes with the Bott homomorphisms $\beta_i$. Thus, $B_\bullet(X)[1/k]$ is an endo-module over the ring $\Z[1/k][t]$, where $t$ acts via $\psi_k$.

\begin{proposition}\label{psiact}
The operation $\psi_k$ acts on the derivative $A_i(X)^{(s)}[1/k]$ via multiplication by $k^{-s-i}$ and on $C_i(X)^{(s)}[1/k]$, $\CH_i(X)[1/k]$ and $K'_0(X)_{(i/i-1)}[1/k]$ via multiplication by $k^{-i}$ for every $i$.
\end{proposition}

\begin{proof}
By Propositions \ref{firstaction} and \ref{secondaction},  $\psi_k$ acts on $A_i(X)[1/k]$ and $C_i(X)[1/k]$ via multiplication by $k^{-i}$. Note that $A_i(X)^{(s)}$ is a submodule of $A_{s+i}(X)$, $C_i(X)^{(s)}$ is a factor module of $C_i(X)$, $\CH_i(X)=C_i(X)^{(1)}$ and $K'_0(X)_{(i/i-1)}$ is a factor module of $\CH_i(X)$.
\end{proof}

It follows from Proposition \ref{psiact} that for every nonzero integer $k$ and every $a\in A_i(X)^{(s)}[1/k]$, we have
\[
k^{-i}\cdot\delta^{(s)}_i(a)=\psi_k(\delta^{(s)}_i(a))=\delta^{(s)}_i(\psi_k(a))=k^{-s-i}\cdot\delta^{(s)}_i(a),
\]
hence every element in
\[
\Im\hspace{0.5mm} [A_i(X)^{(s)}\xra{\delta^{(s)}_i} C_i(X)^{(s)}]=\Ker\hspace{0.5mm} [C_i(X)^{(s)} \tto C_i(X)^{(s+1)}]
\]
is killed by $k^m(k^s-1)$ for some $m\geq 0$.

We consider \emph{supernatural numbers} $k^\infty(k^s-1)$ (see \cite[I.1.3]{Serre97}) and write
\[
N_s:=\gcd k^\infty(k^s-1)
\]
over all $k>1$. For a prime integer $p$ and integer $i>0$ the group $(\Z/p^i\Z)^\times$ is cyclic of order $(p-1)p^{i-1}$ unless $p=2$ and $i\geq 3$ in which case this group is of exponent $2^{i-2}$. It follows that $N_s=2$ if $s$ is odd and
\[
N_s=2^{v_2(s)+2}\cdot\Prod p^{v_p(s)+1},
\]
if $s$ is even, where the product is taken over the set of all prime integers $p$ such that $p-1$ divides $s$ (here $v_p$ is the $p$-adic valuation). For example, $N_2=24$, $N_4=240$, $N_6=2520$,\ldots

We proved the following:

\begin{proposition}\label{maintool}
Let $s$ be a positive integer and $X$ a variety. Then every element in the kernel of the homomorphism $C_i(X)^{(s)} \tto C_i(X)^{(s+1)}$ is killed by $N_s$.\qed
\end{proposition}

Write $\Z_{(p)}$ for the localization of $\Z$ by the prime ideal $p\Z$. Note that if $p$ is a prime divisor of $N_s$, then $p-1$ divides $s$. It follows from Proposition \ref{maintool} that $C_i(X)^{(s)}\tens \Z_{(p)}\tto C_i(X)^{(s+1)}\tens \Z_{(p)}$ is an isomorphism if $p-1$ does not divides $s$. We have proved:

\begin{corollary}(see \cite[Theorem 3.4]{Merkurjev10})\label{divis}
All the differentials in the $s$th derivative of $B_\bullet(X)\tens \Z_{(p)}$ are trivial if $s$ is not divisible by $p-1$.\qed
\end{corollary}

It follows from Proposition \ref{comp} that the kernel of $\varphi_i$ is killed by the product $N_1N_2\cdots N_{d-i-1}$. Every prime divisor $p$ of the product is such that $p-1$ divides an integer $s\leq d-i-1$, hence $p\leq d-i$. We have proved:

\begin{theorem}\label{locallocal}
Let $X$ be a variety of dimension $d$. Then for every $i=0,1,\dots, d$, the map $\varphi_i$ is an isomorphism when localized by $(d-i)!$.
\end{theorem}

\begin{remark}
If $X$ is a smooth variety of dimension $d$, an application of Chern classes and Riemann--Roch theorem imply that $(d-i-1)!\cdot\Ker(\varphi_i)=0$ for every $i>0$ (see \cite[Example 15.3.6]{Fulton98}).
\end{remark}

\begin{proposition}\label{kill}
Let $X$ be a variety. Then the kernel of the Bott homomorphism $\CK_i(X)\to \CK_{i+1}(X)$ is killed by $N_1 N_2\cdots N_{i+1}$ for every $i\geq 0$. In particular, the Bott homomorphism is injective when localized by $(i+2)!$.
\end{proposition}

\begin{proof}
We need to prove that $A_i(X)^{(1)}$ is killed by $N_1 N_2\cdots N_{i+1}$. By induction on $i$ we show that $A_i(X)^{(s)}$ is killed by $N_s N_{s+1}\cdots N_{s+i}$ for every $s\geq 1$. The statement is clear if $i<0$ since $A_i(X)^{(s)}=0$ in this case.

$(i-1)\Rightarrow i$: The factor group $A_i(X)^{(s)}/A_{i-1}(X)^{(s+1)}$ is isomorphic to the kernel of $C_i(X)^{(s)} \tto C_i(X)^{(s+1)}$ and hence is killed by $N_s$ by Proposition \ref{maintool}. By induction, $A_{i-1}(X)^{(s+1)}$ is killed by $N_{s+1}\cdots N_{s+i}$. The result follows.
\end{proof}

\begin{corollary}
Let $X$ be a variety of dimension $d$. Then the associated endo-module $\CK_\bullet(X)$ degenerates when localized by $d!$.
\end{corollary}

\subsection{Direct sum decompositions}

\begin{theorem}\label{th:section}
For every variety $X$ and integer $i\geq 0$, the homomorphism
\[
K_0'(X)_{(i)}[1/(i+1)!] \tto K_0'(X)_{(i/i-1)}[1/(i+1)!]
\]
admits a section, compatibly with proper push-forward homomorphisms.
\end{theorem}
\begin{proof}
For every integer $k >1$, let $r_k = k\cdot\Prod_{j=1}^i (k^j-1) \in \Z[1/(i+1)!]$. If $p> i+1$ is a prime integer and $k>1$ is such that the congruence class $k+p\Z$ is a generator of $(\Z/p\Z)^\times$, then since $p-1>i$, the integer $r_k$ is not divisible by $p$. It follows that the elements $r_k$ for $k>1$ generate the unit ideal in $\Z[1/(i+1)!]$.

Let $M = K_0'(X)_{(i)}[1/(i+1)!]$. For each integer $k >1$, consider the endomorphism
\[
\sigma_k := \Prod_{j=0}^{i-1} \frac{\psi_k - k^{-j}}{k^{-i} - k^{-j}} \colon M[1/r_k] \to M[1/r_k].
\]
Let $N = K_0'(X)_{(i-1)}[1/(i+1)!]$. It follows from Proposition \ref{psiact} that each $\sigma_k$ vanishes on $N[1/r_k]$ and coincides with the identity modulo $N[1/r_k]$. Thus for any $k,k'>1$, we have $\sigma_k = \sigma_k \circ \sigma_{k'}$ and $\sigma_{k'} = \sigma_{k'} \circ \sigma_k$ on $M[1/r_kr_{k'}]$. Since $\sigma_k$ commutes with $\sigma_{k'}$ by \eqref{eq:psi_kk'}, we deduce that $\sigma_k$ and $\sigma_{k'}$ coincide on $M[1/r_kr_{k'}]$. By Zariski descent, there is a unique endomorphism of $M$ whose localization is $\sigma_k$ for each $k>1$. That endomorphism vanishes on $N$ and coincides with the identity modulo $N$, hence induces the required section. The functoriality follows from that of the operations $\psi_k$.
\end{proof}

Theorem \ref{th:section} provides a functorial decomposition
\begin{equation}
\label{eq:decomp_K_i}
K_0'(X)_{(i)}[1/(i+1)!] \simeq \Coprod_{j=0}^i  K_0'(X)_{(j/j-1)}[1/(i+1)!].
\end{equation}
Taking appropriate colimits, we get the following:

\begin{corollary}
Let $X$ be a variety. Then for every $i\geq 0$ there are subgroups $\CK_i(X)^{[j]}\subset \CK_i(X)[1/(j+1)!]$ for all $j=0,1,\dots, i,$ functorial with respect to proper morphisms and such that
\[
\CK_i(X)[1/(i+1)!]=\Coprod_{j=0}^i \CK_i(X)^{[j]}[1/(i+1)!].
\]
Moreover, the localized Bott homomorphism $\CK_{i-1}(X)[1/(i+1)!]\to \CK_i(X)[1/(i+1)!]$ maps $\CK_{i-1}(X)^{[j]}[1/(i+1)!]$ into $\CK_i(X)^{[j]}[1/(i+1)!]$ for all $j=0,1,\dots, i-1$.
\end{corollary}

Combining \eqref{eq:decomp_K_i} with Theorem \ref{locallocal}, we obtain
\begin{corollary}
\label{cor:CH_K0}
If $X$ is a variety of dimension $d$, we have
\[
\CH(X)[1/(d+1)!] \simeq K_0'(X)[1/(d+1)!].
\]
These isomorphisms are compatible with proper push-forward homomorphisms.
\end{corollary}

\begin{remark}
The homomorphism $K_0'(X)_{(d)} \tto K_0'(X)_{(d/d-1)}$ certainly admits a section, since its target is freely generated by the classes $[\cO_Z]$ where $Z$ runs over the $d$-dimensional irreducible components of $X$. Therefore, in fact
\[
\CH(X)[1/d!] \simeq K_0'(X)[1/d!].
\]
However, these isomorphisms are not compatible with proper push-forward homomorphisms in general. For instance, let $X$ be the Severi--Brauer variety of a central division algebra of prime degree $p$ over $F$. Then $d=p-1$ and $K_0'(X) \to K_0'(\Spec F)$ is surjective (as $\chi(X,\cO_X)=1$), but $\CH(X)[1/(p-1)!] \to \CH(\Spec F)[1/(p-1)!]$ is not (because $X$ has no closed point of degree prime to $p$).
\end{remark}

The functoriality in Corollary \ref{cor:CH_K0} implies the following statement (see \cite[Theorem 5.1 (ii)]{invariants}, or \cite[Proposition 1.2]{ELW} for the smooth case):
\begin{corollary}
\label{cor:index_Euler}
Let $X$ be a complete variety of dimension $d$. Then
\begin{itemize}
\item the set of Euler characteristics $\chi(X,\cF)$ of coherent $\cO_X$-modules $\cF$, and

\item the set of degrees of closed points of $X$
\end{itemize}
generate the same ideal in $\Z[1/(d+1)!]$.
\end{corollary}

\begin{remark}
Let $X$ be a complete variety of dimension $d$. Consider the integers
\[
n_X = \gcd_{x \in X_{(0)}} [F(x):F] \;\quad \text{ and } \;\quad d_X = \gcd_{\cF \in \cM(X)} \chi(X,\cF).
\]
Then $d_X \mid n_X$. It follows from Proposition \ref{psiact} that $n_X \mid N_1\cdots N_d \cdot d_X$. Thus if $p$ is a prime number, we have (here $v_p$ is the $p$-adic valuation)
\[
v_p(n_X) \leq v_p(d_X) + \Big\lfloor \frac{d}{p-1}\Big\rfloor  + \sum_{i=1}^{\lfloor d/(p-1) \rfloor} v_p(i).
\]
This bound coincides with that of \cite[Theorem 5.1 (ii)]{invariants} if $d < p(p-1)$, but is not sharp anymore if $d \geq p(p-1)$, at least when $\characteristic F \neq p$ (see \cite[Theorem 5.1 (i)]{invariants}).
\end{remark}

\subsection{Connective \texorpdfstring{$K$}{K}-groups of smooth varieties}

Let $X$ be a smooth variety. We will adopt cohomological notation (upper indices, graded by codimension) and write $\CK^i(X)$, $\CH^i(X)$, $A^i(X)$, $B^i(X)$, $C^i(X)$, etc. The $s$th derivative of $B^\bullet(X)$ will be denoted $B^\bullet(X)_{(s)}$. The graded group $\CK^\bullet(X)$ has a structure of a commutative ring (see \cite{Cai08}). The Bott homomorphisms are multiplications by the \emph{Bott element} $\beta\in \CK^{-1}(X)$. By (\ref{twoiso}), there are canonical ring isomorphisms
\[
\CK^\bullet(X)/(\beta)\simeq \CH^\bullet(X)\quad\text{and} \quad \CK^\bullet(X)/(\beta-1)\simeq K_0(X).
\]

\begin{example}
Let $A$ be a central division algebra of prime degree $p$ over $F$ and $G=\SL_1(A)$ the algebraic group of reduced norm $1$ elements in $A$. Then $K_0(G)=\Z$ (see \cite[Theorem 6.1]{Suslin91a}) and $\CH^*(G)=\Z[\sigma]/(p\sigma,\sigma^p)$, where $\sigma\in \CH^{p+1}(G)$, by \cite[Theorem 9.7]{KM18}. In other words,
\[
\CH^{i}(G)=
\left\{
  \begin{array}{ll}
  \ \Z, & \hbox{if $i=0$;} \\
  (\Z/p\Z)\sigma^j, & \hbox{if $i=(p+1)j$ and $j=1,2,\dots, p-1$;} \\
   \ 0, & \hbox{otherwise.}
  \end{array}
\right.
\]

By Corollary \ref{divis}, all differentials in the $s$th derivative $B^\bullet(G)_{(s)}$ are trivial if $1\leq s<p-1$. It follows that
\[
A^i(G)_{(p-1)}=A^{i-p+2}(G)_{(1)}
\]
for every $i$ and
\[
A^{(p+1)j}(X)_{(p-1)}\subset B^{(p+1)j}(X)_{(p-1)}=\beta^{p-2}B^{(p+1)j}(X)_{(1)}=\beta^{p-2}\CK^{(p+1)j}(X),
\]
\[
C^{(p+1)j}(X)_{(p-1)}=C^{(p+1)j}(X)_{(1)}=\CH^{(p+1)j}(X)=(\Z/p\Z)\sigma^j.
\]
for every $j=1,2,\dots, p-1$.

By \cite[Lemma 3.4]{KM18}, the differential
\[
A^{p+1}(X)_{(p-1)}\to C^{p+1}(X)_{(p-1)}=\CH^{p+1}(X)=(\Z/p\Z)\sigma
\]
is surjective. Choose a pre-image $\theta \in A^{p+1}(X)_{(p-1)}$ of $\sigma$. As $A^{p+1}(X)_{(p-1)}\subset \beta^{p-2}\CK^{p+1}(X)$ we have $\theta=\beta^{p-2}\tau$ for some $\tau\in \CK^{p+1}(X)$. The image of $\tau$ under the natural homomorphism $\CK^{p+1}(X)\to \CH^{p+1}(X)$ is equal to $\sigma$. Since $\theta\in A^{p+1}(X)_{(p-1)}\subset A^{3}(X)_{(1)}$, we have $\beta^{p-1}\tau=\beta\theta=0$ in $\CK^{2}(X)$.

As $\beta\theta\tau^{j-1}=0$, we have $\theta\tau^{j-1}\in A^{(p+1)j}(X)_{(p-1)}$ and the image of $\theta\tau^{j-1}$ under the differential $A^{(p+1)j}(X)_{(p-1)}\to C^{(p+1)j}(X)_{(p-1)}=\CH^{(p+1)j}(X)=(\Z/p\Z)\sigma^j$ is equal to $\sigma^j$.

We proved that the differentials $A^i(G)_{(p-1)}\to C^i(G)_{(p-1)}$ in the $(p-1)$th derivative $B^\bullet(G)_{(p-1)}$ are surjective for all $i>0$. As a consequence, $C^i(G)_{(p)}=0$ for all $i>0$. Since $A^i(G)_{(p)}=0$ for all $i\leq 0$, by Lemma \ref{props}(4), the $p$th derivative $B^\bullet(G)_{(p)}$ degenerates, i.e., $A^i(G)_{(p)}=0$ for all $i$.

It follows that the differentials
\[
A^{i-p+2}(G)_{(1)}=A^i(G)_{(p-1)}\to C^i(G)_{(p-1)}=C^i(G)_{(1)}=\CH^i(G)
\]
are isomorphisms for all $i>0$. As a consequence we get the following calculation:
\[
A^k(G)_{(1)}=
\left\{
  \begin{array}{ll}
    \Z/p\Z, & \hbox{if $k=3+(p+1)j$ for $j=0,1,\dots, p-2$;} \\
    0, & \hbox{otherwise.}
  \end{array}
\right.
\]
It implies that for every $j=0,1,\dots, p-2$ we have a sequence of isomorphisms
\[
\CK^{(p+1)j}(G)\xrightarrow[\sim]{\beta} \CK^{(p+1)j-1}(G)\xrightarrow[\sim]{\beta} \cdots
\xrightarrow[\sim]{\beta}\CK^{3+(p+1)(j-1)}(G)\simeq A^{3+(p+1)(j-1)}(G)_{(1)}=\Z/p\Z.
\]
In particular, the natural homomorphism $\CK^{p+1}(G)\to \CH^{p+1}(G)$ is an isomorphism and hence the element $\tau$ in $\CK^{p+1}(G)$ is unique. Our calculation yields:
\[
\CK^{i}(G)=
\left\{
  \begin{array}{ll}
    \Z, & \hbox{if $i\leq 0$;} \\
    (\Z/p\Z)\beta^k\tau^j, & \hbox{if $i=(p+1)j-k$ for $1\leq j\leq p-1$ and $0\leq k\leq p-2$;} \\
    0, & \hbox{otherwise.}
  \end{array}
\right.
\]
All in all, we have the following formula:
\[
\CK^\bullet(G)=\Z[\beta,\tau]/(p\tau,\tau^p, \beta^{p-1}\tau).
\]
\end{example}

\section{Equivariant connective \texorpdfstring{$K$}{K}-theory}\label{equivar}

Let $G$ be an algebraic group and $X$ a $G$-variety over $F$. Considering the $K$-groups of the categories of $G$-equivariant coherent $\cO_X$-modules with support of bounded dimension (see \cite[\S 3]{KM19b}) one gets an exact couple leading to a BGQ type spectral sequence and an endo-module $B_\bullet(G,X)$ with the first derivative groups $\CK_\bullet(G,X)$ the equivariant connective $K_0$-groups of $X$. The endo-module $B_\bullet(G,X)$ is stable but it is not bounded below in general.

In the case $X=\Spec(F)$ we write $\CK^\bullet(\B G)$ for $\CK^\bullet(G,X)$, $\CH^\bullet(\B G)$ for $\CH^\bullet(G,X)$, etc. The category of $G$-equivariant coherent $\cO_X$-modules in this case is the category of finite dimensional representations of $G$ and hence $K'_0(\B G)$ coincides with the representation ring $R(G)$ of $G$. In particular, we have surjective homomorphisms
\[
\varphi^i:\CH^i(\B G)\tto R(G)^{(i/i+1)},
\]
where $\CH^i(\B G)$ are the equivariant Chow groups (defined by Totaro in \cite{Totaro99}). The (topological) filtration on $R(G)$ was defined in \cite{KM19c}. The following example illustrates how the calculation of equivariant connective groups $\CK^i(\B G)$ allows us to determine the differentials in the endo-module.

Assume that $\ch(F)\neq 2$ and $G=\O^+_n$ the split \emph{special orthogonal} group of odd degree $n$. It is known (see \cite[Theorem 5.1]{RAA06}) that
\[
\CH(\B G)=\Z[c^{\CH}_2,c^{\CH}_3,\dots ,c^{\CH}_n]/(2c^{\CH}_{odd}).
\]
and
\[
R(G)=\Z[c^K_2,c^K_4,\dots, c^K_{n-1}],
\]
where $c^{\CH}$ and $c_i^K$ are the classical and $K$-theoretic Chern classes respectively.
The term $R(G)^{(i)}$ of the topological filtration on $R(G)$ is generated by monomials in the Chern classes of degree at least $i$. The homomorphism
\[
\varphi^*:\CH^*(\B G)\tto R(G)^{(*/*+1)}
\]
takes $c^{\CH}_i$ to the class of $c^{K}_i$ if $i$ is even and to $0$ if $i$ is odd. In particular, $\Ker(\varphi^*)$ is generated by $c^{\CH}_{i}$ with $i\geq 3$ odd (see \cite[Example 5.3]{KM19c}).

The same reasoning to prove that the ring $\CH(\B G)$ is generated by Chern classes in \cite[\S 15]{Totaro99} can be applied to show that $\CK(\B G)$ is also generated by $\CK$-theoretic Chern classes $c_1,c_2,\dots, c_n$. We determine relations between Chern classes.

\begin{lemma}
\label{chernformula}
Let $E$ be a rank $r$ vector bundle over a variety $X$. For any $i \in \Z$ and $j\in \{0,\dots,r\}$, we have, as homomorphisms $\CK_i(X) \to \CK_{i-j}(X)$
\[
c_j(E^\vee) = [\det E] \cdot \sum_{l=j}^r (-1)^l\binom{l}{j} \beta^{l-j} c_l(E).
\]
\end{lemma}
\begin{proof}
For any rank $s$ vector bundle $M$ over a variety $Y$, consider the homomorphism
\[
\rho(M) = [\det M] \cdot \sum_{l=0}^s c_l(M) (-1-\beta)^l \colon \CK(Y) \to \CK(Y).
\]
The conclusion of the lemma may be reformulated as $c(E^\vee) = \rho(E)$. If $0 \to M' \to M \to M'' \to 0$ is an exact sequence of vector bundles of constant ranks, then $c(M^\vee) = c(M'^\vee)\circ c(M''^\vee)$ and $\rho(M) = \rho(M') \circ \rho(M'')$. Thus by the splitting principle, we may assume that $r=1$. Since $\CK_i(X)$ is generated by the images of the push-forward homomorphisms $\CK_i(Z) \to \CK_i(X)$, where $Z \subset X$ is closed subscheme of dimension at most $i$, we may assume that $\dim X \leq i$. Then the homomorphism $\beta \colon \CK_{i-1}(X) \to \CK_i(X) = K_0'(X)$ is injective, hence it will suffice to prove that
\[
\id = [E] \cdot (\id - c_1(E)) \quad \text{ and } \quad c_1(E^\vee) = -[E] \cdot c_1(E)
\]
as endomorphisms of $K_0'(X)$. This follows at once from the formula $c_1(L) = 1 -[L^\vee] \in K_0(X)$, valid for any line bundle $L$ over $X$ (in particular for $L=E$ and $L=E^\vee$).
\end{proof}

For every $i\geq 1$, let $Q^i$ be the subgroup of $\CK^i(\B\O^+_n)$ generated by $\beta^j c_{i+j}$ over all $j\geq 0$. Write $Q_{even}^i$ for the subgroup of $Q^i$ generated by $\beta^j c_{i+j}$ with $i+j$ even. Obviously, $\beta Q^i\subset Q^{i-1}$ and $\beta Q_{even}^i\subset Q_{even}^{i-1}$.

\begin{proposition}\label{tilde}
For every odd $i=1,3,\ldots, n$,
\begin{enumerate}
  \item $Q^{i-1}=Q_{even}^{i-1}$,
  \item There is an element $\tilde c_i\in c_i+ Q_{even}^i$ such that $2\tilde c_i=0$ and $\beta\tilde c_i=0$.
\end{enumerate}
\end{proposition}

\begin{proof}
We proceed by descending induction on $i$. Let $i=n$.
It follows from Lemma \ref{chernformula} that $c_n=-c_n$ and $c_{n-1}=c_{n-1}-n\beta c_{n}$, i.e., $n\beta c_{n}=0$. Setting $\tilde c_n=c_n$ we deduce that $2c_n=0$ and $\beta c_n=0$ since $n$ is odd.

The group $Q_{even}^{n-1}$ is generated by $c_{n-1}$ and $Q^{n-1}$ is generated by $c_{n-1}$ and $\beta c_n=0$, hence $Q^{n-1}=Q_{even}^{n-1}$.

$(i+2)\Rightarrow i$: It follows from Lemma \ref{chernformula} and the induction hypothesis that
\[
2c_i\in \beta Q^{i+1}=\beta Q_{even}^{i+1}= Q_{even}^{i},
\]
thus $2c_i=\sum_{even\ j>i} a_j\beta^{j-i}c_j$ with $a_j\in\Z$. Mapping to $R(G)$ we see that $2c^K_i=\Sum a_jc^K_j$ in $R(G)$. On the other hand, $c^K_i\in R(G)=\Z[c^K_2,c^K_4,\dots, c^K_{n-1}]$, hence all $a_j$ are even, therefore, $2c_i\in 2Q_{even}^{i}$. We deduce that there is $\tilde c_i\in c_i+ Q_{even}^i$ such that $2\tilde c_i=0$.

Lemma \ref{chernformula} for $c_{i-1}$ yields $i\beta c_{i}\in \beta^2 Q^{i+1}=\beta^2 Q_{even}^{i+1}\subset Q_{even}^{i-1}$ and therefore, $i\beta \tilde c_{i}\in  Q_{even}^{i-1}$. As $Q_{even}^{i-1}$ maps injectively to $R(G)$ and $\beta \tilde c_{i}$ maps to zero (since $\beta \tilde c_{i}$ is $2$-torsion and $R(G)$ is torsion-free), we have $i\beta \tilde c_{i}=0$. But $i$ is odd, hence $\beta \tilde c_{i}=0$.

It follows from $\tilde c_i\in c_i+ Q_{even}^i$ and $\beta \tilde c_{i}=0$ that $\beta c_i\in \beta Q_{even}^i\subset Q_{even}^{i-1}$. Finally,
\[
Q^{i-1} = \Z c_{i-1}+ \Z \beta c_i+ \beta^2 Q^{i+1}= \Z c_{i-1}+ \Z \beta c_i+ \beta^2 Q_{even}^{i+1} \subset Q_{even}^{i-1}.\qedhere
\]
\end{proof}

Note that since the Bott map $\CK^1(\B G)\to R(G)$ is injective and $R(G)$ is torsion free, the element $\tilde c_{1}$ is trivial. It follows from Proposition \ref{tilde} that the ring $\CK(\B G)$ is generated by $c_2, \tilde c_{3},c_4,\dots, \tilde c_{n}$ and $\beta$. Write $\tilde c_{i}=c_i$ for all even $i$.

Under the natural homomorphism $\CK(\B G)\to \CH(\B G)$ the class $\tilde c_i$ goes to $c^{\CH}_i$. It is immediate that the natural diagonal homomorphism $\CK(\B G)\to R(G)\times \CH(\B G)$ is injective. This implies two things: first the relations $2\tilde c_i=0$ and $\beta\tilde c_i=0$ are the defining relation between the $\tilde c_{i}$'s. In other words,
\[
\CK(\B G)=\Z[\tilde c_2,\tilde c_3,\dots, \tilde c_n,\beta]/(2\tilde c_{odd},\beta \tilde c_{odd}).
\]

Second, the derived endo-module of $B_\bullet(\B G)$ degenerates, i.e., all nonzero differentials appear in the first derivative only and therefore, the second derivative degenerates.


\begin{thebibliography}{10}

\bibitem{Cai08}
{\sc Cai, S.}
\newblock Algebraic connective {$K$}-theory and the niveau filtration.
\newblock {\em J. Pure Appl. Algebra 212}, 7 (2008), 1695--1715.

\bibitem{ELW}
{\sc Esnault, H., Levine, M., and Wittenberg, O.}
\newblock Index of varieties over {H}enselian fields and {E}uler characteristic
  of coherent sheaves.
\newblock {\em J. Algebraic Geom. 24}, 4 (2015), 693--718.

\bibitem{Fulton98}
{\sc Fulton, W.}
\newblock {\em Intersection theory}, second~ed., vol.~2 of {\em Ergebnisse der
  Mathematik und ihrer Grenzgebiete. 3. Folge. A Series of Modern Surveys in
  Mathematics [Results in Mathematics and Related Areas. 3rd Series. A Series
  of Modern Surveys in Mathematics]}.
\newblock Springer-Verlag, Berlin, 1998.

\bibitem{GS-87}
{\sc Gillet, H., and Soul\'{e}, C.}
\newblock Intersection theory using {A}dams operations.
\newblock {\em Invent. Math. 90}, 2 (1987), 243--277.

\bibitem{ega-2}
{\sc Grothendieck, A.}
\newblock \'{E}l\'{e}ments de g\'{e}om\'{e}trie alg\'{e}brique. {II}. \'{E}tude
  globale \'{e}l\'{e}mentaire de quelques classes de morphismes.
\newblock {\em Inst. Hautes \'{E}tudes Sci. Publ. Math.}, 8 (1961), 222.

\bibitem{Har-Alg-77}
{\sc Hartshorne, R.}
\newblock {\em Algebraic geometry}.
\newblock Springer-Verlag, New York, 1977.
\newblock Graduate Texts in Mathematics, No. 52.

\bibitem{invariants}
{\sc Haution, O.}
\newblock Invariants of upper motives.
\newblock {\em Doc. Math. 18\/} (2013), 1555--1572.

\bibitem{2nd}
{\sc Haution, O.}
\newblock Detection by regular schemes in degree two.
\newblock {\em Algebr. Geom. 2}, 1 (2015), 44--61.

\bibitem{KM18}
{\sc Karpenko, N.~A., and Merkurjev, A.~S.}
\newblock Motivic decomposition of compactifications of certain group
  varieties.
\newblock {\em J. Reine Angew. Math. 745\/} (2018), 41--58.

\bibitem{KM19c}
{\sc Karpenko, N.~A., and Merkurjev, A.~S.}
\newblock Chow filtration on representation rings of algebraic groups.
\newblock {\em Preprint\/} (2019).

\bibitem{KM19b}
{\sc Karpenko, N.~A., and Merkurjev, A.~S.}
\newblock Equivariant connective $k$-theory.
\newblock {\em IHES Preprints IHES/M/19/15\/} (2019).

\bibitem{Levine97}
{\sc Levine, M.}
\newblock Lambda-operations, {$K$}-theory and motivic cohomology.
\newblock In {\em Algebraic {$K$}-theory ({T}oronto, {ON}, 1996)}, vol.~16 of
  {\em Fields Inst. Commun.} Amer. Math. Soc., Providence, RI, 1997,
  pp.~131--184.

\bibitem{LM07}
{\sc Levine, M., and Morel, F.}
\newblock {\em Algebraic cobordism}.
\newblock Springer Monographs in Mathematics. Springer, Berlin, 2007.

\bibitem{Merkurjev10}
{\sc Merkurjev, A.}
\newblock Adams operations and the {B}rown-{G}ersten-{Q}uillen spectral
  sequence.
\newblock In {\em Quadratic forms, linear algebraic groups, and cohomology},
  vol.~18 of {\em Dev. Math.} Springer, New York, 2010, pp.~305--313.

\bibitem{RAA06}
{\sc Molina~Rojas, L.~A., and Vistoli, A.}
\newblock On the {C}how rings of classifying spaces for classical groups.
\newblock {\em Rend. Sem. Mat. Univ. Padova 116\/} (2006), 271--298.

\bibitem{Quillen73}
{\sc Quillen, D.}
\newblock Higher algebraic {$K$}-theory. {I}.
\newblock In {\em Algebraic {$K$}-theory, {I}: {H}igher {$K$}-theories ({P}roc.
  {C}onf., {B}attelle {M}emorial {I}nst., {S}eattle, {W}ash., 1972)\/} (1973),
  pp.~85--147. Lecture Notes in Math., Vol. 341.

\bibitem{Serre97}
{\sc Serre, J.-P.}
\newblock {\em Galois cohomology}.
\newblock Springer-Verlag, Berlin, 1997.
\newblock Translated from the French by Patrick Ion and revised by the author.

\bibitem{Soule85}
{\sc Soul{\'e}, C.}
\newblock Op\'erations en {$K$}-th\'eorie alg\'ebrique.
\newblock {\em Canad. J. Math. 37}, 3 (1985), 488--550.

\bibitem{Suslin91a}
{\sc Suslin, A.~A.}
\newblock ${K}$-theory and ${K}$-cohomology of certain group varieties.
\newblock In {\em Algebraic $K$-theory}, vol.~4 of {\em Adv. Soviet Math.}
  Amer. Math. Soc., Providence, RI, 1991, pp.~53--74.

\bibitem{Totaro99}
{\sc Totaro, B.}
\newblock The {C}how ring of a classifying space.
\newblock In {\em Algebraic {$K$}-theory ({S}eattle, {WA}, 1997)}, vol.~67 of
  {\em Proc. Sympos. Pure Math.} Amer. Math. Soc., Providence, RI, 1999,
  pp.~249--281.

\bibitem{Weibel94}
{\sc Weibel, C.~A.}
\newblock {\em An introduction to homological algebra}, vol.~38 of {\em
  Cambridge Studies in Advanced Mathematics}.
\newblock Cambridge University Press, Cambridge, 1994.

\end{thebibliography}
\bibliographystyle{acm}

\end{document}